\documentclass[12pt, A4]{article}
\usepackage{amssymb}
\usepackage{amsmath,amsfonts,amsthm,amstext}
\usepackage[english]{babel}
\usepackage[latin1]{inputenc}
\usepackage{makeidx}
\usepackage{amscd}
\usepackage[dvips]{graphicx}
\usepackage{fancybox}
\usepackage{niceframe}
\usepackage{times}
\usepackage{mathrsfs}
\usepackage{bbm}
\usepackage[all]{xy}
\usepackage{fancyhdr}
\usepackage{sectsty}
\usepackage[Conny]{fncychap}
\usepackage[symbol]{footmisc}
\usepackage[T1]{fontenc}
\DefineFNsymbols{stars}[text]{* {**} {*} {\S} {\I}}
\setfnsymbol{stars}

\newtheorem{thm}{Theorem}[section]
\newtheorem{cor}[thm]{Corollary}
\newtheorem{pro}[thm]{Proposition}
\newtheorem{deff}[thm]{Definition}
\newtheorem{lem}[thm]{Lemma}
\newtheorem{rem}[thm]{Remark}

\newcommand{\nc}{\newcommand}

\nc{\cc}{\D{C}} \nc{\hh}{\D{H}} \nc{\nn}{\D{N}} \nc{\oo}{\D{O}}
\nc{\qq}{\D{Q}}
 \nc{\rr}{\D{R}}
\nc{\zz}{\D{Z}} \nc{\livre}{\ast}

\def\G{{\cal G}}
\def\Vightarrow#1{\smash{\mathop{\longrightarrow}\limits^{#1}}}

\nc{\barr}{\begin{array}} \nc{\earr}{\end{array}}
\nc{\bthm}{\begin{thm}} \nc{\ethm}{\end{thm}}
\nc{\bpro}{\begin{pro}} \nc{\epro}{\end{pro}}
\nc{\blem}{\begin{lem}} \nc{\elem}{\end{lem}}
\nc{\bins}{\begin{ins}} \nc{\eins}{\end{ins}}
\nc{\bcor}{\begin{cor}} \nc{\ecor}{\end{cor}}
\nc{\brem}{\begin{rem}} \nc{\erem}{\end{rem}}
\nc{\bdeff}{\begin{deff}} \nc{\edeff}{\end{deff}}
\newtheorem*{theorem*}{Theorem}
\nc{\bea}{\begin{eqnarray}} \nc{\eea}{\end{eqnarray}}
\nc{\D}[1]{{\mathbb#1}}
\def\R{\rm I\kern -.2em R}
\def\N{\rm I\kern -.18em N}
\def\Z{\rm Z\kern -.332em Z}
\def\de{\rm [\kern -.15em [}
\def\dd{\rm ]\kern -.15em ]}
\def\||{\hspace{0.15cm}|\hspace{0.15cm}}

\def\dbigcup{\mathinner{\bigcup \mkern -13.2mu \rlap{\raise 0.6ex\hbox{.}}\mkern 14.9mu}}
\title{Hyperbolic 3-Manifolds  Groups are Subgroup Conjugacy Separable}
\author{S. C. Chagas,
 \,\,\, P. A. Zalesskii
 \footnote{\vspace*{-.5cm} Both authors were supported by
CNPq.}}
\begin{document}
%====================================================================

%\author{S. C. Chagas and P. A. Zalesskii}
%\address{Universidade de Brasília\\
%Campus Universitário Darcy Ribeiro\\
%70910-900 Brasília\\
%Brazil }

%\email{pz@mat.unb.br}

%\subjclass{MSC: 20E26; 20E18}

%\keywords{Subgroup Conjugacy Separability, Conjugacy Separability}

%\date{}
%----------additions
%\dedicatory{}
%%% ----------------------------------------------------------------------

\maketitle

\begin{abstract}
 A group $G$ is called subgroup conjugacy separable
 if for
every pair of non-conjugate finitely generated subgroups  of $G$,
 there exists a finite quotient of $G$ where
the images of these subgroups are not conjugate.
  It is proved
that  the fundamental group of  a hyperbolic  3-manifold (closed or with cusps) is subgroup conjugacy separable.
\end{abstract}

\section{Introduction}

O. Bogopolski and F. Grunewald \cite{BG}  introduced an
important notion of subgroup conjugacy separability for a group
$G$. A group $G$ is said to be subgroup conjugacy separable if for
every pair of non-conjugate finitely generated subgroups $H$ and
$K$ of $G$,
 there exists a finite quotient of $G$ where
the images of these subgroups are not conjugate.  Thus the subgroup conjugacy separability  is a residual property of groups, which logically continues the following classical  residual  properties of groups: the residual finiteness, the conjugacy separability, and the subgroup separability (LERF).
A.I. Mal'cev was the first, who noticed, that finitely presented residually finite (resp. conjugacy separable) groups have solvable word problem (resp. conjugacy problem) \cite{M-58}. Arguing in a similar way, one can show that finitely presented subgroup separable  groups have solvable membership problem and that finitely presented subgroup conjugacy separable groups have solvable conjugacy problem for finitely generated subgroups. The last means, that there is an algorithm, which given a finitely presented subgroup conjugacy separable group $G = \langle X \mid R\rangle$ and two finite sets of elements $Y$  and $Z$, decides whether the subgroups $\langle Y\rangle$ and $\langle Z\rangle$ are conjugate in $G$.

Bogopolski and Grunewald proved that
free groups and the fundamental groups of finite trees of finite
groups subject to a certain normalizer condition, are subgroup
conjugacy separable.  For finitely generated virtually free groups the result was proved in  \cite{CZ-15}. Also, O. Bogopolski and K-U. Bux in \cite{BB}
proved that surface groups are conjugacy subgroup separable.
In \cite{CZ-16} the authors of the present paper extended this result to
 limit groups.

The objective of this paper is to show that the fundamental group of a hyperbolic 3-manifold is subgroup conjugacy separable. Bogopolski and Bux posed it as an open question on page 3 in \cite{BB-16}.

\begin{thm}  \label{3-manifolds} The fundamental group $\pi_1M$ of a  hyperbolic 3-manifold $M$ (closed or with cusps) is subgroup conjugacy separable.
\end{thm}

 Note that the fundamental group $\pi_1M$ of a  hyperbolic 3-manifold $M$  is subgroup  separable (see  \cite[Corollary 5.5 (1)]{AFW-15}), a crucial property used   in the proof.  In fact,  the proof is valid  for  hyperbolic subgroup separable virtually special groups.
 
 \begin{thm} \label{SSHVS} A  hyperbolic subgroup separable virtually special group $G$ is infinite subgroup conjugacy separable\end{thm}

 These theorems were known before only for quasiconvex subgroups. In  \cite[Theorem 1.2]{CZ-16},   the authors of this paper proved that quasiconvex subgroups of hyperbolic virtually special groups  are subgroup conjugacy separable (Bogopolsky and Bux  gave an independent proof of this result under the complementary torsion freeness assumption).  Note also that Theorem \ref{SSHVS} valid also for finite  soluble  subgroups of $G$ (see Theorem \ref{HVS}).

 Virtually special groups own its importance  due  to  Daniel
Wise who proved in \cite{W} that 1-relator groups with torsion are
virtually special, answering positively a question of Gilbert
Baumslag who asked in \cite{B-67} whether these groups are
residually finite.
In fact, many  groups of geometric origin are
virtually special: the fundamental group of a hyperbolic
3-manifold (Agol \cite{A-13}), small cancellation groups
(a combination of \cite{W} and \cite{A-13}) and hyperbolic Coxeter
groups (Haglund and Wise \cite{HW-2010}). 
 
 % In fact, for torsion free hyperbolic virtually special groups subgroup conjugacy %separability implies conjugacy separability as explained in Remark \ref{scs implies %cs}.

\subsection*{Acknowledgments}

The authors would like to express their deep gratitude to Ashot Minasyan for his comments and suggestions that  led to  substantial  improvement of the paper.

\section{Preliminaries}\label{pelim} In this section we introduce the reader to conceptions and terminology of the profinite version of the Bass-Serre theory of  groups acting on trees used in the paper.

We consider the following standard definitions.
 Our graphs are oriented  graphs. A graph $\Gamma$ is a set together with
a distinguished subset of {\it vertices} $V= V(\Gamma)$ and together
with two maps $d_0, d_1: \Gamma\longrightarrow V$, which are the
identity when restricted to $V$. This graph is called {\it
profinite} if $\Gamma$ is a profinite  space (i.e., a compact,
Hausdorff and totally-disconnected topological space), $V$ is a
closed subset of $\Gamma$, and the mappings $d_i$ are continuous.
If $e\in \Gamma$, we say that $d_0(e)$ and $d_1(e)$ are the origin
and terminal vertex of $e$, respectively.  The complement $E=
E(\Gamma)= \Gamma-V(\Gamma)$ of $V(\Gamma)$ in $\Gamma$ is called
the set (space) of {\it edges} of $\Gamma$. For basic concepts
such as connectedness, or of when a graph is a tree, see
\cite[Chapter I]{DD-89}, or   \cite[Part I]{Serre},   for
abstract graphs. We assume that the reader is familiar with basic notions of Bass-Serre theory of groups acting on trees treated in these books. 

We also assume that the reader knows basic facts about profinite groups, in particular the notion of the profinite topology on a group that can be found in \cite[Chapter 3]{RZ-10}. Following the tradition of combinatorial group theory  a subgroup  $H$ of a group $G$ will be called {\it separable}  if it is closed in the profinite topology of $G$.

For a profinite space $X$  that is the inverse limit of finite discrete spaces $X_j$, $[[\widehat{\mathbb{Z}} X]]$ is the inverse limit  of $\widehat 
{[\mathbb{Z}}X_j]$, where $[\widehat{\mathbb{Z}} X_j]$ is the free $\widehat{\mathbb{Z}}$-module  with basis $X_j$. For a pointed profinite space $(X, *)$
that is the inverse limit of pointed finite discrete spaces $(X_j, *)$, $[[\widehat{\mathbb{Z}} (X, *)]]$ is the inverse limit  of $
[\widehat{\mathbb{Z}} (X_j, *)]$, where $[\widehat{\mathbb{Z}} (X_j, *)]$ is the free $\widehat{\mathbb{Z}}$-module  with basis
$X_j \setminus \{ * \}$ \cite[Chapter~5.2]{RZ-10}.

Given a profinite graph $\Gamma$ define the pointed space $(E^*(\Gamma), *)$ as  $\Gamma / V(\Gamma)$ with the image of $V(\Gamma)$ as a distinguished point $*$.
By definition  a profinite tree  $\Gamma$ is a profinite graph with a short exact sequence
$$
0 \to [[\widehat{\mathbb{Z}}(E^*(\Gamma), *)]] \Vightarrow\delta [[\widehat{\mathbb{Z}} V(\Gamma)]] \Vightarrow{\epsilon} \widehat{\mathbb{Z}} \to 0
,$$
where $\delta(\bar{e}) = d_1(e) - d_0(e)$ for every $e \in E(\Gamma)$, $\bar{e}$ the image of $e$ in $E^*(\Gamma)$ and $\epsilon(v) = 1$ for
every $v \in V(\Gamma)$.

We refer for further details of the profinite version of the Bass-Serre theory to  \cite{ZM}.
 If $v$  and $w$ are vertices of a tree  (respectively, of a profinite tree)
 $\Gamma$, we denote by $[v,w]$ the smallest subtree  (respectively, a
 profinite subtree) of $\Gamma$ containing $v$ and $w$.

 A group $H$ is said to act on a graph $\Gamma$ if it acts on $\Gamma$ as a set and if in
 addition $d_i(hm)= hd_i(m)$, for all $h\in H$  and $m\in \Gamma$  ($i=0,1$); if $\Gamma$
 is a profinite graph and $H$ a profinite group, we assume that the action is continuous.
 The quotient $ \Gamma/H$  inherits a natural graph structure (respectively, profinite graph structure).
 
\medskip 

Let $G_1$ and $G_2$ be profinite groups with a common closed
subgroup $D$; the {\it profinite free amalgamated product}
$G_1\amalg_HG_2$ is  the push-out $G$ of $G_1$ and $G_2$ over $H$
in the category of profinite groups; if the canonical
homomorphisms $G_1\longrightarrow G$ and $G_2\longrightarrow G$
are embeddings, one says that $G$ is proper (see 
\cite[Chapter 9]{RZ-10} for more details). Note that if $G$ is residually finite then $G$ is proper.

Let $H$ be a profinite group and let $f:A\longrightarrow B$ be a continuous
isomorphism between closed subgroups $A,B$ of $H$. A  {\it profinite
HNN-extension} of $H$  with associated
subgroups $A,B$ consists of a profinite group
$G={\rm HNN}(H,A,t)$, an element $t\in
G$, and a continuous homomorphism
$\varphi:H\longrightarrow G$ with $t(\varphi (a))t^{-1}= \varphi f(a)$  and satisfying the following universal
property: for any profinite group $K$, any $k\in K$ and any continuous homomorphism
$\psi:H\longrightarrow K$ satisfying $k(\psi(a))k^{-1}=\psi f(a)$ for all
$a\in A$, there is a unique continuous homomorphism
$\omega:G\longrightarrow K$ with
$\omega(t)=k$ such that  the diagram
 $$\xymatrix{G\ar@{.>}[dr]^{\omega}&\\
                              H\ar[u]^{\varphi}\ar[r]^{\psi}&K}$$ is commutative.
We shall refer to $\omega$ as the homomorphism induced by $\psi$.

\medskip Observe that one needs to test the above universal property only for
finite  groups $K$, for then it holds automatically for any profinite group $K$,  since
$K$ is an inverse limit of finite groups.

We define the standard tree $S(G)$ on which $G$ acts  (respectively, $S(\widehat{G})$ on which the profinite completion $\widehat{G}$
  acts)  for the cases of an amalgamated free product $G=G_1\ast_HG_2$ (respectively, $\widehat G=\widehat G_1\amalg_{\widehat{H}}\widehat{G}_2$)
  and an HNN-extension $G=HNN(G_1, H,t)$ (respectively, $\widehat G=HNN(\widehat G_1,\widehat H,t)$) since we shall use them frequently  for these cases.

  \begin{itemize}
  \item Let $G=G_1\ast_HG_2$. Then
 the vertex set is $V(S(G))= \displaystyle G/G_1\cup G/G_2$,
  the edge set is $E(S(G))= G/H$, and
  the initial and terminal vertices of an edge $gH$ are
  respectively  $gG_1$ and $gG_2$.
   \item Similarly, let $\widehat G=\widehat G_1\amalg_{\widehat{H}}\widehat{G}_2$. Then the vertex set is $V(S(\widehat{G}))=
  \displaystyle \widehat G/\widehat G_1 \cup\widehat{G}/\widehat{G}_2$,
  the edge set is $E(S(\widehat{G}))= \widehat{G}/\widehat{H}$, and
  the initial and terminal vertices of an edge $ g\widehat{H}$ are
  respectively  $g\widehat{G}_1$ and $g\widehat{G}_2$.
    \item Let $G=HNN(G_1, H,t)$. Then
 the vertex set is $V(S(G))= \displaystyle G/G_1$,
  the edge set is $E(S(G))= G/H$, and
  the initial and terminal vertices of an edge $gH$ are
  respectively  $gG_1$ and $gtG_1$.
  \item Similarly let $\widehat G=HNN(\widehat G_1, \widehat H,t)$. Then
 the vertex set is $V(S(\widehat G))= \displaystyle \widehat G/\widehat G_1$,
  the edge set is $E(S(\widehat G))= \widehat G/\widehat H$, and
  the initial and terminal vertices of an edge $g\widehat H$ are
  respectively  $g\widehat G_1$ and $gt\widehat G_1$.
  \end{itemize}
 
 The tree $S(G)$
  naturally embeds in $S(\widehat G)$ if and only if the subgroups $H$, $G_1$ and $G_2$   are separable in  $G$, or equivalently $H$ is closed in $G_1$  (and in $G_2$ in the case of amalgamation)  with
respect to the topology induced by the profinite topology on $G$ (see 
\cite[Proposition 2.5]{CB-13}).

These constructions  are particular cases of the general construction of the profinite  fundamental group of a finite graph of profinite groups.  
 
 When we say that ${\cal G}$ is a finite graph of profinite groups we mean that it contains the data of the
underlying finite graph, the edge profinite groups, the vertex profinite groups and the attaching continuous maps. More precisely,
let $\Delta$ be a connected finite graph. A    graph of profinite groups $({\cal G},\Delta)$ over
$\Delta$ consists of a specifying profinite group ${\cal G}(m)$ for each $m\in \Delta$, and continuous monomorphisms
$\partial_i: {\cal G}(e)\longrightarrow {\cal G}(d_i(e))$ for each edge $e\in E(\Delta)$. The  fundamental group
$$\Pi= \Pi_1({\cal G},\Delta)$$
of the graph of profinite groups $({\cal G},\Delta)$ is defined by means of a universal property: $\Pi$ is a profinite group together
with the following data  and conditions:
\begin{enumerate}
\item [(i)] a maximal subtree $T$ of $\Delta$;

\smallskip
\item [(ii)]  a collection of continuous homomorphisms
$$\nu_m: {\cal G}(m)\longrightarrow \Pi\quad (m\in \Delta), $$
  and     a continuous  map
   $E(\Delta) \longrightarrow  \Pi$, denoted $e\mapsto t_e$  ($e\in E(\Delta)$), such that
$t_e=1$, if $e\in E(T)$, and
$$(\nu_{d_0 (e)}\partial_0)(x)= t_e(\nu_{d_1 (e)}\partial_1)(x)t_e^{-1},\quad  \forall x\in {\cal G}(e), \ e\in E(\Delta); $$

\smallskip
\item [(iii)]  the following universal property is satisfied:

\medskip
\noindent whenever one has the following data

\begin{itemize}
\item $H$ is a profinite group,\\
\item $\beta_m: {\cal G}(m)\longrightarrow \Pi\quad (m\in \Delta)$
a collection of continuous homomorphisms,\\
\item a map $e\mapsto s_e$ ($e\in E(\Delta)$)  with $s_e=1$, if
$e\in E(T)$, and\\
\item $(\beta_{d_0 (e)}\partial_0)(x)= s_e(\beta_{d_1
(e)}\partial_1)(x)s_e^{-1}, \forall x\in {\cal G}(e), \ e\in
E(\Delta),  $\\
\smallskip
\noindent then there exists a unique continuous homomorphism $\delta : \Pi\longrightarrow  H$ such that $\delta(t_e)= s_e$
 $(e\in E(\Delta))$, and for each $m\in\Delta$ the diagram
\end{itemize}

\medskip

$$\xymatrix{&
\Pi  \ar[dd]^\delta   \\  {\cal G}(m)  \ar[ru]^{\nu_m}
\ar[rd]_{\beta_m }\\ &H }$$

\medskip
\noindent commutes.
\end{enumerate}

In \cite[paragraph (3.3)]{ZM},  the fundamental group
 $\Pi$ is  defined explicitly in terms of generators and relations.  It is also proved there
that the definition given above is independent of the choice of
the maximal subtree $T$.
We use the notation $\Pi(m) = {\rm Im}(\nu_m)$; so $\Pi(m)\cong {\cal G}(m)$, for $m\in \Delta$.

Associated with the graph of groups $({\cal G}, \Delta)$ there is
a corresponding  {\it standard profinite graph} (or universal covering graph)
  $S=S(\Pi)=\dbigcup
\Pi/\Pi(m)$.  The vertices of
$S$ are those cosets of the form
$g\Pi(v)$, with $v\in V(\Delta)$
and $g\in \Pi$; the incidence maps of $S$ are given by the formulas:

$$d_0 (g\Pi(e))= g\Pi(d_0(e)); \quad  d_1(g\Pi(e))=gt_e\Pi(d_1(e)) \,  (\,e\in E(\Delta)).  $$

 In fact $S$  is a profinite tree (cf. \cite[Theorem 3.8]{ZM}.
 There is a natural  action of
 $\Pi$ on $S$, and clearly $ S/\Pi= \Delta$.
 
 \begin{rem}\label{completion}
 If $\pi_1(\G,\Gamma)$ is the fundamental group of a finite graph of groups then one has the induced graph of profinite completions of edge  and vertex groups $(\widehat\G,\Gamma)$ and  a natural homomorphism $\pi=\pi_1(\G,\Gamma)\longrightarrow \Pi_1(\widehat\G,\Gamma)$. It is an embedding if $\pi_1(\G,\Gamma)$ is residually finite. In this case $\Pi_1(\widehat \G,\Gamma)=\widehat{\pi_1(\G,\Gamma)}$ is simply the profinite completion. 
Moreover, 
 
 \begin{enumerate}
 
 \item[(i)]  The tree $S(\pi)$
  naturally embeds in $S(\widehat\pi)$ if and only if  the edge and vertex groups $\G(e)$, $\G(v)$     are separable in  $\pi_1(\G,\Gamma)$, or equivalently $\G(e)$ are closed in $\G(d_0(e))$,   $\G(d_1(v))$    with
respect to the topology induced by the profinite topology on $\pi$ (see
\cite[Proposition 2.5]{CB-13}).

\item[(ii)]  If $H$ is an infinite finitely generated subgroup of $\pi$ then  by   combination of Theorem 4.12 and Proposition 4.13  of Chapter 1 in \cite{DD-89} there exists a minimal $H$-invariant subtree $T_H$ of $S(\pi)$ and 
it is unique. Moreover, $T_H/H$ is finite.

\item[(iii)]  If  $S(\pi)$
  naturally embeds in $S(\widehat \pi)$,  the closure $\overline T_H$ in $S(\widehat \pi)$ is a $\overline H$-invariant profinite subtree and by  \cite[Lemma 1.5]{RZ-10} contains a unique (in $S(\widehat \pi)$) minimal $\overline H$-invariant subtree  $\widehat T_{\overline H}$. Moreover,  $\widehat T_{\overline H}/\overline H$ is finite since it is a subgraph of a quotient graph $\overline T_H/\overline H$ of  the finite graph $T_H/H$.
\end{enumerate}

\end{rem}

\begin{lem}\label{free H-action}  Within the hypotheses of Remark \ref{completion}  (iii)  suppose $\pi$ is subgroup separable  and  $H$ acts freely on $S(\pi)$. Then $\overline T_H=\widehat T_H$ and $T_H$ is a connected component of $\overline T_H$ (considered as a usual graph).\end{lem}

\begin{proof}  For a graph $\Delta$  denote by $D_\Delta$ a maximal subtree of $\Delta$.   Since $H$ acts freely on $S(\pi)$ it is free of rank $(T_H/H)\setminus D_{T_H/H}$.  Since $G$ is subgroup separable $\overline H\cong \widehat H$ is a free profinite groups of the same rank as $H$.  By \cite[Lemma 2.8]{RZ-96} $\overline H$ acts on $S(\widehat \pi)$ freely as well. By  \cite[Proposition 2.11]{Z-89} $\overline H$ is a free profinite group of rank $(\overline T_H/\overline H)\setminus D_{\overline T_H/\overline H}$ and since $T_H/H$ is a covering of $\overline T_H/\overline H$ (because of the free action of $\overline H$) we deduce that $T_H/H=\overline T_H/\overline H$.  To see that $\overline T_H=\widehat T_H$ let $\Sigma$ be a connected transversal of $\widehat T_H/\overline H$ in $S(\pi)$ with $d_0(\Sigma)\subseteq \Sigma$ and $\Omega$ its maximal connected subtree. Put $K=\langle k_e\in \pi\mid  k_ed_1(e)\in\Sigma\setminus \Omega \rangle$. Then $K$ is a free group freely generated by $\{k_e\in \pi\mid  k_ed_1(e)\in\Sigma\}$ and also  $\overline H$ is freely generated (as a profinite group) by  $\{k_e\in \pi\mid  k_ed_1(e)\in\Sigma\}$ (see \cite[Lemma 2.3]{Z-89}). It follows that $\widehat H\cong \overline H=\overline K\cong \widehat K$.  Since $\overline T_H/\overline H$ contains $ \widehat T_H/\overline H$ by Remark \ref{completion} (iii), $H=K* L$ is a free product for some $L$ and so $\widehat H=\widehat K\amalg  \widehat L$ the profinite free product. It follows from $\widehat H\cong\widehat K$ that $\widehat L=1=L$ so that  $K=H$.  Then by the minimality of $H$-invariant subtree $T_H$, we have $T_H=H\Sigma$ and so $\overline T_H=\overline{H\Sigma}=\widehat T_H$.

 If $T_H$ is not a connected component of $\overline T_H$ then there exists an edge $e\in \overline T_H\setminus T_H$ with an incident vertex $v\in T_H$. Since $\overline T_H/\overline H=T_H/H$,   $\bar h e\in T_H$ for some $\bar h\in \overline H$ and so $\bar h v \in T_H$. Hence there exists $h\in H$ with $hv=\bar hv$ and since the action of $\overline H$ on $S(\widehat \pi)$ is free we have $\bar h=h$ implying $e\in S(H)$, a contradiction.

\end{proof}

The following term will be important in the following section to perform an induction on hierarchy.

\begin{deff} We say that a residually finite group $G$ is adjustable if for any pair of  finitely generated  subgroups $A$ and $B$ of $G$ such that $\overline A^{\gamma}=\overline B$ for some $\gamma\in \widehat G$  there exists  $\beta \in \overline B$, such that $A^{\gamma\beta}\cap B\neq 1$.\end{deff}

\begin{rem}\label{virtually adjustable} Note that adjustability is preserved by commensurability. Indeed, a finite index subgroup of an adjustable group is clearly adjustable. 

If  a group $H$ has a finite index adjustable subgroup $G$, then $H$ is adjustable. Indeed, passing to core we may assume that $G$ is normal.  Suppose  $A, B$ are finitely generated  subgroups  of $G$ such that $\overline A^{\gamma}=\overline B$ for some $\gamma\in \widehat G$ and $A$, $B$ are finite then they coincide with their closures and there is nothing to show. If on the other hand they are infinite then $(\overline A\cap \widehat G)^{\gamma}=\overline B\cap \widehat G$ and since $G$ is adjustable,  there exists  $\beta \in \overline B\cap \widehat G$, such that $A^{\gamma\beta}\cap B\neq 1$  as needed.\end{rem}

\begin{pro}\label{fundamental adjustable}  Let $G=\pi_1(\G,\Gamma)$ be the fundamental group of a finite graph of  finitely generated adjustable groups. Suppose $G$ is subgroup separable.  Then $G$ is adjustable. 
\end{pro}

\begin{proof} Let $A,B$ be infinite  finitely generated  subgroups of $G$ such that $\overline A^{\gamma}=\overline B$ for some $\gamma\in \widehat G$ (if $A$ and $B$ are finite there is nothing to prove).  Let $S$ be a standard tree on which $G$ acts and  let  $T_A$ and $T_B$ be the minimal $A$-invariant and $B$-invariant subtrees of the standard tree $S$ on which $G$ acts (see Remark \ref{completion} (ii)).  Denote by $\widehat{T}_{\overline A}$ and $\widehat T_{\overline B}$  the (unique) minimal $\overline H_1$ and $\overline H_2$-invariant profinite subtrees in 
$\widehat{S}$ respectively (see Remark \ref{completion} (iii)).  Then, $\gamma^{-1} \widehat T_{\overline A}= \widehat T_{\overline B}$, by the uniqueness of  the minimal $\overline B$-invariant subtree
$\widehat T_{\overline B}$ in $\widehat{S}$. 

\medskip

(i) If the action of $A$ on $S$ is free, then by \cite[Lemma 2.8]{RZ-96} $\overline A, \overline B$ and $B$ have trivial edge and vertex  stabilizers as well. Hence by \cite[Proposition 1.6]{RZ-14} $\overline{T}_A=\widehat T_{\overline A}$, $\overline{T}_B=\widehat T_{\overline B} $ and  by  Lemma \ref{free H-action} $T_A$ and $T_B$ are the (usual) connected component of $\widehat T_A$ and $\widehat{T}_B$ respectively. Since $\overline T_B=\overline B T_B$ this means that $\beta^{-1}\gamma^{-1} T_A= T_B$ for some $\beta \in \overline B$. It follows that  $A^{\gamma\beta}= B$, since $T_B/A^{\gamma\beta}=\overline{T_B}/\widehat A^{\gamma\beta}=\overline{T_B}/\widehat B=T_B/B$ by  \cite[Proposition 1.6]{RZ-14}.  

\medskip

(ii) If $A_w\neq 1$ for some $w\in S$, then $A_v\neq 1$ for some $v\in T_A$. Since $\gamma^{-1} \widehat{T}_{\overline A}= \widehat{T}_{\overline B}$ and $\widehat{T}_{\overline B}/\overline B$ is a subgraph of a  quotient of  $T_B/B$ (see Remark \ref{completion}(iii)), we have $\beta^{-1}\gamma^{-1} v\in T_B$ for some $\beta \in \overline B$.   Since $S(G)/G=S(\widehat G)/\widehat G$, the vertices   $v,\beta\gamma v$ are in the same $G$-orbit and so there exists $g\in G$ with $gv=\beta^{-1}\gamma^{-1} v$ so that $\gamma \beta g\in \widehat G_v$. Therefore $\overline A_v^{\gamma\beta g}=(\overline A\cap \widehat G_v)^{\gamma\beta g}=  \overline A^{\gamma\beta g}\cap \widehat G_v=\overline B^g\cap \widehat G_v=(\overline B^g)_v$.   Since $G_v$ is adjustable there exists  $\beta_v^g\in (\overline B^g)_v$   such that  $A_v^{\gamma\beta g\beta_v^{g}}=A_v^{\gamma\beta\beta_v g}= (B^g)_v\leq B^g$.  Then $A_v^{\gamma\beta\beta_v}\leq  B$ and since $\beta\beta_v\in \overline B$   the result is proved. \end{proof}

\begin{rem}\label{free action} The case (i) of the proof of Proposition \ref{fundamental adjustable} shows that if the action of $A$ on $S$ is free then in addition we have $\beta^{-1}\gamma^{-1} T_A= T_B$ and   $A^{\gamma\beta}= B$. 

On the other hand if the action of $A$ on $S$ is not free the case (ii) of the proof shows that the element $a\neq 1$ with $a^{\gamma\beta}\in B$ exists in every non-trivial vertex stabilizer $A_v$.
\end{rem}

\section{General results}

A subgroup $H$ of a group $G$ is called a virtual retract if $H$
is a semidirect factor (retract) of some finite index subgroup of $G$.  A group $G$ is called hereditarily conjugacy separable if every finite index 
subgroup of $G$ is conjugacy separable. 

\medskip
We begin  this section with  the  key

\begin{lem}\label{elementsconjugate} Let $G$ be an adjustable  conjugacy separable group and $A,B$ be finitely generated separable   subgroups of $G$. Suppose there exists an element $ a\in A$  such that $ C_{\widehat G}(a)G=(\overline A\cap C_{\widehat G}(a))G $.  Then the conjugacy of\, $\overline A$ and $\overline B$ in $\widehat G$ implies  the conjugacy of $A$ and $B$  in $G$.  In particular, the statement holds if   $[C_G(a):\langle a\rangle]$ is finite and $\overline{C_G(a)}=C_{\widehat G}(a)$; the latter equality holds for every $1\neq a\in G$ if $G$ is hereditary conjugacy separable .\end{lem}

\begin{proof}  Suppose $\overline A^{\gamma}=\overline B$ for some $\gamma$ in $\widehat G$. Since $G$ is adjustable   $a^{\gamma\beta}\in B$ for some $1\neq a\in A$, $\beta\in \overline B$. Since $G$ is conjugacy separable $a^g=a^{\gamma\beta}$  for some $g\in G$, so replacing $B$ with 
$B^{g^{-1}}$ and $\gamma$ with $\gamma\beta g^{-1}$ we may assume 
$\gamma\in C_{\widehat G}(a)$. By hypothesis, $ C_{\widehat G}(a)G=(\overline A\cap C_{\widehat G}(a))G $,   so $\gamma=a'g$ for some $a'\in \overline A\cap C_{\widehat G}(a)$, $g\in G$  and therefore once more replacing $B$ with 
$B^{g^{-1}}$ and $\gamma$ with $\gamma g^{-1}$  we may assume that  $\gamma \in \overline A\cap C_{\widehat G}(a)$. It follows then that $\overline A=\overline B$ and since $A$ and $B$ are separable we deduce that $A=B$.

To prove the last statement note that since  $[C_G(a): \langle a\rangle]$ is finite  and   $C_G(a)$ is dense in $C_{\widehat G}(a)$, then  $C_{\widehat G}(a)=\overline{\langle a\rangle} C_G(a)$ and so   $C_{\widehat G}(a)G= (\overline A\cap C_{\widehat G}(a))G $ clearly holds.  We conclude the proof observing that by \cite[Proposition 3.2]{M}   hereditarily conjugacy separability of $G$ implies   $C_{\widehat G}(a)=\overline{C_G(a)}$ for every $a\in G$.
 \end{proof}
 
 \begin{thm}\label{general} Let $G$ be an adjustable,   subgroup separable and hereditarily conjugacy separable group. Suppose that  for any element $1\neq g\in G$    the index $[C_G(g):\langle g\rangle]$ is finite. Then $G$ is subgroup conjugacy separable.\end{thm}

\begin{proof}  Let $A,B$ be finitely generated subgroups of $G$ such that  $\overline A^{\gamma}=\overline B$ for some $\gamma$ in $\widehat G$.
 Then all the premises of  
 Lemma \ref{elementsconjugate}  are satisfies   and so applying it we deduce that  $A$ and $B$ are conjugate in $G$ as required.\end{proof} 
 
Since the centralizer of a non-trivial  element in a torsion free hyperbolic group is cyclic  (see  \cite[Proposition 3.5]{Mih}) we deduce the following 
 
 \begin{cor}  A torsion free  adjustable,   subgroup separable and hereditarily conjugacy separable hyperbolic group is subgroup conjugacy separable.\end{cor}

\begin{deff}
The class of  groups with a \emph{ hierarchy} is the smallest class of groups, closed under isomorphism, that contains the trivial group, and such that, if
\begin{enumerate}
\item $G=A*_CB$ and $A$, $B$ each have a  hierarchy, or
\item $G=A*_C$ and $A$ has a  hierarchy,
\end{enumerate}
 then $G$ also has a hierarchy.
\end{deff}

Groups with hierarchy allow to use induction on their hierarchy. Thus we can deduce from Proposition \ref{fundamental adjustable} the following 

\begin{pro}  \label{adjustable}  A subgroup separable group with hierarchy is adjustable.\end{pro}

\begin{thm}\label{up}  A hyperbolic  hereditarily conjugacy separable  group $H$ having a finite index subgroup separable subgroup $G$ with hierarchy is infinite subgroup conjugacy separable.\end{thm}

\begin{proof}  By Proposition \ref{adjustable} and Remark \ref{virtually adjustable} $H$ is adjustable, and since subgroup separability passes to  overgroups of finite index, is subgroup separable. 

 Let $H_1$, $H_2$ be infinite  finitely generated subgroups of $H$ such that $\overline H_1^{\gamma}=\overline H_2$ for some $\gamma\in \widehat H$.

Since a residually finite hyperbolic group is virtually torsion free (see  \cite[Theorem 5.1]{KW-00}) $H$ contains a torsion free finite index subgroup $K$ so replacing $G$ by $G\cap K$  we may assume that $G$ is torsion free. Then $H_1$ possesses an element of infinite order in $G$.  The centralizer of an element $h$ of infinite order in a hyperbolic group is virtually cyclic (see  \cite[Proposition 3.5]{Mih}) and so $h$ generates the subgroup of finite index in its centralizer. Thus  by Lemma \ref{elementsconjugate} $H_1$ and $H_2$ are conjugate in $H$.
\end{proof}

A group $G$ is called virtually  compact special if there exists a special
compact cube complex X having a finite index subgroup of $G$ as
its fundamental group (see \cite{W} for definition of special cube
complex). Since the hyperbolic  fundamental group of such a complex admits a hierarchy \cite{haglund_combination_2012} we deduce from Proposition \ref{adjustable} and Remark \ref{virtually adjustable}  the following

\begin{cor}  A hyperbolic virtually compact special group is adjustable.\end{cor}

\begin{lem}\label{centralizer} Let $A$ be a finite subgroup  of a hyperbolic virtually compact special group $G$. Then
\begin{enumerate} 

\item[(i)] $C_G(A)$ is a virtual retract of $G$ and  is virtually compact special.

\item[(ii)] $C_G(A)$ is dense in $C_{\widehat G}(A)$.

\end{enumerate}

\end{lem}

\begin{proof}  (i)  Since the group $G$ is hyperbolic, it is well-known that centralizers of elements in $G$ are quasiconvex
(see, for example, \cite[Ch. III.$\Gamma$, Prop. 4.14]{B-H}) and is also hyperbolic (cf. \cite[Lemma 3.8]{Mih}). Then inductively on the number of elements using $C_G(a)\cap C_G(b)=C_{C_G(a)}(b)$ we deduce that   the centralizer of any finite subgroup of a hyperbolic group is quasiconvex and hyperbolic.    In \cite[Corollary  7.8]{HW-2008} Haglund and Wise proved that
any quasiconvex subgroup of $G$ is virtually compact special and  in \cite{HW-2010} that it is virtual retract of $G$. Thus the centralizer of a finite subgroup in $G$ is a virtual retract of $G$.

(ii) Using (i)  we prove  (ii)  by induction on $|A|$ . Let $K$ be a maximal subgroup of $A$ and $a\in A\setminus K$. Then $C_G(A)=C_{C_G(K)}(a)$  and  $C_{\widehat G}(A)=C_{C_{\widehat G}(K)}(a)$.  By induction hypothesis  $C_G(K)$ is dense in $C_{\widehat G} (K)$ and  $C_{C_G(K)}(a)$ is dense in $C_{\widehat{C_G(K)}}(a)$. Then using  that the profinite topology of $G$ induces the full profinite topology on virtual retracts we have $\overline{C_G(A)}=\overline{C_{C_G(K)}(a)}=C_{\widehat{C_G(K)}}(a)=C_{\overline{C_G(K)}}(a)=C_{C_{\widehat G}(K)}(a)=C_{\widehat G}(A)$ as required.
\end{proof} 

%Now Theorem \ref{SSHVS} announced in the introduction follows as a special %case  from Theorem \ref{up} and hereditarily conjugacy separability of a virtually %special group \cite{MZ-15}. 

We are in a position now to prove the main general result of this paper.

\begin{thm}  \label{HVS}  A  hyperbolic subgroup separable virtually compact special   group  $G$ is   infinite subgroup conjugacy separable. If all finite subgroups of $G$ are soluble then $G$ is subgroup conjugacy separable. \end{thm}

\begin{proof} Note first that by \cite[Theorem 1.1]{MZ-15} a virtually compact special group is hereditarily  conjugacy separable. Let $H_1$, $H_2$ be infinite  finitely generated subgroups of $G$ such that $\overline H_1^{\gamma}=\overline H_2$ for some $\gamma\in \widehat H$. If $H_1, H_2$ are infinite the result follows from Theorem \ref{up}. 

\medskip
Assume $H_1, H_2$ are finite. We shall use induction on the order $|H_1|=|H_2|$.
 The conjugacy separability of $G$ implies  the result for $H_1$  cyclic of order $p$. 

Suppose now $|H_1|>p$ and  let $A$ be a maximal proper normal subgroup of $H_1$.  Since the $C_G(A)$ has finite index in $N_G(A)$ and by Lemma \ref{centralizer} is virtual retract of $G$ and in particular is finitely generated,  we deduce from subgroup separability  of  $G$   existence  of a finite index normal subgroup $U$ of $G$ such that $U\cap N_G(A)\leq C_G(A)$. Since $\widehat G=\widehat UG$ replacing $H_2$ by its conjugate in $G$ we may assume that $\gamma\in \widehat U$. Moreover,  since $G$ is virtually torsion free we may assume that $U$ is torsion free. Then $H_2\leq H_1\overline U\cap G= H_1U$ and so we may assume that $G=UH_1=U\rtimes H_1$. By induction hypothesis $A^{\gamma g}=A$ for some $g\in G$ so replacing $H_2$ by $H_2^{g^{-1}}$ we may assume that $\gamma\in C_{\widehat U}(A)=N_{\widehat U}(A)$ and so $A\leq H_1\cap H_2$. Since $C_G(A)$ is virtual retract and  virtually compact special by Lemma \ref{centralizer} (i) so is $N_G(A)$,  and since $C_G(A)$  is dense in $C_{\widehat G}(A)$ by Lemma \ref{centralizer}(ii),  we have $\overline{ N_G(A)}=\overline{C_U(A)}\rtimes H_1=C_{\widehat G}(A)\rtimes H_1=N_{\widehat G} (A)$.  In particular, the induced profinite topology on $N_G(A)$ is the full profinite topology and so $\overline{N_G(A)}=\widehat{N_G(A)}$. Thus we may assume that $G=N_G(A)$ and so $A$ to be normal in $G$. 

Since $|H_1/A|<  H_1$ by the induction hypothesis $H_1/A^{gA}=H_2/A$ for some $g\in G$. Then $H_1^g=H_2$ as needed.
\end{proof} 

%\begin{rem} The proof of Theorem \ref{HVS} completes also the proof of Theorem 2.6 %in \cite{CZ-16} where the case of finite subgroups was left out.\end{rem}

%\begin{deff} A subgroup $H$ of a group $G$ is called malnormal if %$H^g\cap H= 1$ for every $g\not\in H$.\end{deff} 

\section{Manifolds}

Here we apply the general result of the previous section to closed and cusp hyperbolic 3-manifolds. For closed 3-manifolds the result follows quickly.

\begin{thm}
The fundamental group $\pi_1M$ of a  closed hyperbolic  3-manifold $M$  is subgroup conjugacy separable.
\end{thm}

\begin{proof}  In this case $\pi_1M$ is   hyperbolic. By result of Agol \cite{A-13} $\pi_1M$ is virtually compact special and subgroup separable.  It is also hereditarily conjugacy separable (see  \cite[G8]{AFW-15}). Thus the result follows from Theorem \ref{HVS}.
\end{proof}

We consider the cusped case now. Recall that a subgroup of $\pi_1M$  is called {\it peripheral}  if it is conjugate to the fundamental group of a cusp and so is isomorphic to $\Z\times \Z$.  

It is well-known that $\pi_1M$ is relatively hyperbolic to peripheral subgroups (\cite[Theorem 5.1]{F-98}).
  We refer the reader to \cite{hruska_relative_2010} for a survey of the various equivalent definitions of relative hyperbolicity.

%Note also that  \cite[Theorem 9.1]{WZ-16}) the fundamental group of a cusped %hyperbolic 3-manifold of finite volume  admits hierarchy.  

\begin{thm}
The fundamental group $H=\pi_1M$ of a  cusped hyperbolic 3-manifold $M$  is subgroup conjugacy separable.
\end{thm}

\begin{proof}  
  The group  $H$ is subgroup separable ( \cite[Corollary 5.5]{AFW-15})  and hereditarily conjugacy separable (see in \cite[G8 ]{AFW-15}).   Since by \cite[Theorem 9.1]{WZ-15}  $H$ admits a hierarchy   by Proposition \ref{adjustable} combined with Remark \ref{virtually adjustable} it is  adjustable.

\medskip
Let $A$, $B$ be  finitely generated subgroups of $H$ such that $\overline A^{\gamma}=\overline B$ for some $\gamma\in \widehat H$.  Note that by \cite[Proposition 3.2]{M} hereditary conjugacy separability of $G$ implies $\overline{C_G(a)}=C_{\widehat G}(a)$ for any $a$.

 If $A$ is not contained in a peripheral subgroup then since  $H$ is relatively hyperbolic to   peripheral subgroups (\cite[Theorem 5.1]{F-98}) there exists $a\in A$ such that $C_G(a)$ is infinite cyclic (cf. \cite[Theorem 4.3]{O-06}). Hence $[C_G(a): \langle a\rangle]$ is finite.  Therefore we deduce from Lemma \ref{elementsconjugate} that  $A$ and $B$ are conjugate.

If $A$ is contained in a peripheral subgroup $P$ then  it is either free abelian of rank 2 and so of finite index in $P$ or cyclic. In the first case the condition  $C_{\widehat G}(a)G=(\overline A\cap C_{\widehat G}(a))G $ is satisfied  so by Lemma \ref{elementsconjugate}  $A$ and $B$ are conjugate. If $A$ and $B$ are cyclic, then since $H$ is adjustable there exists $1\neq a\in A$ and $\beta\in \overline B$  with $a^{\gamma\beta} \in B$ and since $H$ is conjugacy separable $a=a^{\gamma\beta h}$ for some $h$ in $H$. Then conjugating $B$ by $h^{-1}$  we may assume that $a=a^{\gamma\beta}$ and since peripheral subgroups   pairwise  intersect trivially (cf. \cite[Lemma 4.7]{HWZ-13} this implies that $B\leq P$. But $P$ is free abelian, so cyclic subgroups $A,B$ intersecting non-trivially must coincide.
\end{proof}

\bigskip
{\it Author's Adresses:}

\medskip
Sheila C. Chagas\\
Departamento de Matem\'atica,\\
~Universidade de Bras\'\i lia,\\
70910-900 Bras\'\i lia DF,\\
Brazil

sheila@mat.unb.br

\medskip
Pavel A. Zalesski\\
Departamento de Matem\'atica,\\
~Universidade de Bras\'\i lia,\\
70910-900 Bras\'\i lia DF\\
Brazil

pz@mat.unb.br

\end{document}